\documentclass[12pt]{amsart}
\usepackage[dvips]{color}
\usepackage{amsmath}
\usepackage{amsxtra}
\usepackage{amscd}
\usepackage{amsthm}
\usepackage{amsfonts}
\usepackage{amssymb}
\usepackage{eucal}
\usepackage{epsfig}
\usepackage{graphics}
\usepackage{ytableau}
\usepackage{pb-diagram}
\usepackage{fourier}
\usepackage{tikz}
\newcommand{\R}{{\mathbb R}}
\newcommand{\C}{{\mathbb C}}
\newcommand{\Z}{{\mathbb Z}}

\newcommand{\nn}{\nonumber}
\newtheorem{thm}{Theorem}[section]
\newtheorem{prop}[thm]{Proposition}
\newtheorem{lem}[thm]{Lemma}
\newtheorem{rem}[thm]{Remark}
\newtheorem{cor}[thm]{Corollary}

\numberwithin{equation}{section}

\allowdisplaybreaks

\begin{document}
\begin{title}{Connection problem for the generalized hypergeometric function}
\end{title}
\author{ Y. Matsuhira and H. Nagoya}
\address{YM: School of Mathematics and Physics, Kanazawa University, Kanazawa, Ishikawa 920-1192, Japan}
\email{y.matsu0727@gmail.com}
\address{HN: School of Mathematics and Physics, Kanazawa University, Kanazawa, Ishikawa 920-1192, Japan}
\email{nagoya@se.kanazawa-u.ac.jp}

\begin{abstract} 
We solve connection problem  between fundamental  solutions at 
singular points $0$ and $1$ for the generalized hypergeometric function, 
using analytic continuation of the integral representation. 
All connection coefficients are products of the sine and the cosecant. 
\end{abstract}

\maketitle

\section{Introduction}

Let $n \in \Z$ ($n>0$), $\alpha_1, \dots, \alpha_{n+1}, \beta_1, \dots, \beta_n \in \C$. 
The generalized hypergeometric series is 
\begin{align*}
 _{n+1} F _n \left( \begin{matrix}
\alpha_1 ,  \dots ,  \alpha_{n+1} \\
\beta_1 ,  \dots ,   \beta_n \\
\end{matrix} ; z \right) =\sum_{k=0}^\infty \frac{(\alpha_1)_k \cdots (\alpha_{n+1})_k} {(\beta_1)_k \cdots (\beta_n)_k k!}z^k, 
\end{align*}
where $(a)_k = a(a+1) \cdots (a+k-1)$. The generalized hypergeometric series converges on $|z|<1$ 
and 
satisfies the Fuchsian differential equation with three singular points $0$, $1$, $\infty$:
\begin{equation*}
\left\{ \frac{d}{dz}  \prod_{k=1}^{n}\left(\frac{d}{dz} + \beta_k -1\right) 
 -z \prod_{k=1}^{n+1}\left(\frac{d}{dz}+\alpha_k\right) \right\}F=0. 
\end{equation*}
We call this differential equation 
the generalized hypergeometric equation and 
 a solution to the  equation a generalized hypergeometric 
function. 

The Riemann scheme, which is the table of the characteristic exponents, 
of the generalized hypergeometric equation 
 is 
\begin{align*}
\left\{
 \begin{array}{ccc}
      z=0 & z=1 &  z=\infty \\
      0 & 0 & \alpha_1 \\
      1-\beta_1 & 1 & \alpha_2 \\
      \vdots & \vdots  & \vdots \\
      1-\beta_{n-1} & n-1 &  \alpha_n \\
      1-\beta_n & \sum_{i=1}^{n}\beta_i-\sum_{i=1}^{n+1}\alpha_i & \alpha_{n+1}
    \end{array}
\right\}. 
\end{align*}
Assume $\alpha_i-\beta_j\notin\Z$ and $\beta_i-\beta_j\notin\Z$ ($i\neq j$).  
Then the generalized hypergeometric equation 
admit a fundamental system of solutions 
at $z=0$ given by 
\begin{align*}
f_i^{(0)}(z)=(-z)^{1-\beta_i} {}_{n+1} F _n \left( \begin{matrix}
\alpha_1-\beta_i+1 ,  \alpha_2-\beta_i+1, \dots ,  \alpha_{n+1}-\beta_i+1 \\
\beta_1-\beta_i+1 ,  \dots ,  \widehat{\beta_i-\beta_i+1} ,  \dots ,  \beta_{n+1}-\beta_i+1 \\
\end{matrix} ; z \right)\quad (i=1,\ldots,n+1), 
\end{align*}
where the symbol $\widehat{A}$ means omitting $A$ and $\beta_{n+1}=1$, 
and a  fundamental system of solutions 
at $z=\infty$ given by 
\begin{align*}
f_i^{(\infty)}(z)=(-z)^{-\alpha_i}  {}_{n+1} F _n \left( \begin{matrix}
\alpha_i-\beta_1+1 ,  \alpha_i-\beta_2+1, \dots ,  \alpha_i-\beta_{n+1}+1 \\
\alpha_i-\alpha_1+1 ,  \dots ,  \widehat{\alpha_i-\alpha_i+1} ,  \dots ,  \alpha_i-\alpha_{n+1}+1 \\
\end{matrix} ; \frac{1}{z} \right)\quad (i=1,\ldots,n+1). 
\end{align*}
Fundamental systems of solutions at $z=1$ consist of 
one non-holomorphic solution with the characteristic exponent 
$\sum_{i=1}^{n}\beta_i-\sum_{i=1}^{n+1}\alpha_i$ 
and $n$ holomorphic solutions. 

Connection problem between fundamental systems of solutions 
of the generalized hypergeometric equation 
has been solved by various authors by 
several methods \cite{Kawabata 1}, \cite{Kawabata 2}, 
\cite{Kawabata 3}
\cite{Mimachi intersection numbers for n+1Fn}, 
\cite{Norlund}, 
\cite{Okubo Takano Yoshida}, \cite{S}, 
\cite{Winkler}. 
The choice of the fundamental systems of solutions 
\begin{equation*}
X_0=\left\{ f_1^{(0)}(z),\ldots, f_{n+1}^{(0)}(z)\right\},\quad 
X_\infty=\left\{ f_1^{(\infty)}(z),\ldots, f_{n+1}^{(\infty)}(z)\right\}
\end{equation*}
at $z=0$ and $z=\infty$ is canonical. 
All connection coefficients associated with $X_0$ and $X_\infty$ are products 
of the Gamma function and the inverse of the Gamma function. 
We note that if we multiply $f_i^{(0)}(z)$ and $f_i^{(\infty)}(z)$ by suitable scalars, 
then connection coefficients become products 
of the sine and the cosecant. 

On the other hand, there is no canonical choice of 
 fundamental systems of solutions at $z=1$. 
A connection matrix depends on choice of the fundamental system of 
solutions and connection coefficients are not necessarily 
products 
of the Gamma function and the inverse of the Gamma function. 
If we take the fundamental system of solutions at $z=1$ by 
$(z-1)^i+O((z-1)^n)$ ($i=0,1,\ldots,n-1$), $(z-1)^a(1+O(z-1))$ 
with $a=\sum_{i=1}^{n}\beta_i-\sum_{i=1}^{n+1}\alpha_i$, 
then the connection coefficients with $X_0$ or $X_\infty$ involve values 
of the generalized hypergeometric series at $z=1$. 
Kawabata reported in \cite{Kawabata 3} that 
there is a fundamental system of solutions at $z=1$ such that 
non-diagonal elements of the connection matrix with $X_0$ or $X_\infty$ are 
products of the sine and the cosecant and diagonal elements of that 
are one minus products of the sine and the cosecant. 

To the author's knowledge, connection 
problem between fundamental  systems of solutions at 
singular points $0$ and $1$ for the generalized hypergeometric equation 
has not been solved {\it completely}. 
The aim of the present note is to do that.  
Namely, we give connection matrices whose elements 
are products of the sine and the cosecant. 
In order to obtain the connection formula expressing a fundamental system of solution 
at $z=1$ in terms of the integral representations corresponding to $X_0$, 
we calculate analytic continuation of the connection formula expressing 
the integral representations corresponding to $X_\infty$  
in terms of the integral representations corresponding to $X_0$ 
\cite{Mimachi intersection numbers for n+1Fn}. 
We deform the domains of integration loaded with the integrand of 
the generalized hypergeometric function.  
The technique dealing with  deformation for the loaded cycles associated with 
Selberg type integrals 
is elaborated in \cite{Mimachi  Iwahori-Hecke}. 

Unfortunately, by using 
 analytic continuation of the connection formula expressing 
the integral representations corresponding to $X_0$  
in terms of the integral representations corresponding to $X_\infty$ 
\cite{Mimachi intersection numbers for n+1Fn}, 
we have not found a way to obtain 
the connection formula expressing the integral representations corresponding to $X_0$  
 in terms of the fundamental system of solution at $z=1$. 
Instead, we present directly the inverse matrix of the connection matrix 
expressing a fundamental system of solution 
at $z=1$ in terms of the integral representations corresponding to $X_0$,
 and give a proof by the residue calculus. 

There had been two methods deriving connection formulas of 
rigid Fuchsian systems by hypergeometric integrals. One is to use 
the Cauchy's integral theorem for obtaining linear relations among loaded cycles 
\cite{Aomoto 1}, \cite{Aomoto 2}. 
Such method was used for obtaining the connection formulas of the Simpson's Even four 
\cite{Haraoka Mimachi}, ${}_3F_2$ \cite{Mimachi 3F2},  sixth order rigid Fuchsian systems derived from conformal field theory \cite{BHS}.  
However, in the case of a general $n$-multiple hypergeometric integrals, 
there are too many linear relations among loaded cycles 
to obtain connection formulas between fundamental systems of solutions. 
Another method is to compute intersection numbers of loaded cycles 
\cite{Mimachi intersection numbers for n+1Fn}, which avoids to solve 
too many linear relations among loaded cycles. 
We believe that to compute analytic continuation of known connection formulas 
is also useful for obtaining another connection formulas. Especially because 
the connection coefficients for multiplicity-free case, 
such as the connection coefficients for elements of $X_0$ and $X_\infty$, 
is solved in \cite{Oshima}. 

The plan of the paper is as follows. 
In Section 2, we prepare notations and recall the 
connection matrix of the generalized hypergeometric  function between $X_0$ and $X_\infty$ 
in \cite{Mimachi intersection numbers for n+1Fn}. 
In Section 3,  we introduce domains of integration for a fundamental system of solutions 
at $z=1$, and prove the connection formulas between 
the fundamental system of solutions at $z=1$ and $z=0$. 
At the end, we remark on periodicity of the connection matrices.

\section{Preliminary}

We consider a multi-valued function
\begin{align*}
u(t)=\prod_{i=1}^n t_i^{\lambda_i} \prod_{i=1}^{n} (t_{i}-t_{i-1})^{\mu_i}
(t_n-z)^{\mu_{n+1}}
\end{align*}
for $t=(t_1,\ldots,  t_n)\in\C^n$ with parameters $\lambda_i,\mu_i\in\C$ 
defined on 
\begin{align*}
T_z=\C^n-\bigcup_{i=1}^n \{t_i=0\} \cup \bigcup_{i=1}^{n+1} \{t_{i-1}-t_i=0\}, 
\end{align*}
where 
\begin{align*}
t_0=1,\quad t_{n+1}=z. 
\end{align*}

The function $u(t)$ is the integrand of an integral representation of the generalized 
hypergeometric series ${}_{n+1}F_n(z)$.  Namely, we have  
\begin{align}\label{eq_integral_representation_for_GHS}
{}_{n+1}F_n\left( \begin{matrix}
\alpha_1 ,  \dots ,  \alpha_{n+1} \\
\beta_1 ,  \dots ,   \beta_n \\
\end{matrix} ; z \right)=&\prod_{1\le i\le n}B(\alpha_i,\beta_i-\alpha_i)
 \int_{D^{(0)}_{n+1}} u(t)dt,
\end{align}
where $B(\alpha,\beta)$ is the Beta function and 
$D^{(0)}_{n+1}=\{t\in T_z\mid 1<t_1<t_2\cdots<t_n\}$ with 
\begin{equation}\label{eq_transformation_lambda_alpha_beta}
\lambda_i=\alpha_{i+1}-\beta_i\quad (1 \leq i \leq n), \quad 
\mu_i=\beta_i-\alpha_i-1\quad (1 \leq i \leq n+1).
\end{equation}
Here, we suppose 
\begin{align*}
\mathrm{Re}(\alpha_i)>0,\quad \mathrm{Re}(\beta_i-\alpha_i)>0 \quad (1\le i\le n)
\end{align*}
for the convergence of the integral and fix the arguments as 
\begin{align*}
\mathrm{arg}(t_i)=\mathrm{arg}(t_i-t_{i-1})=0\quad (1\le i\le n),
\quad |\mathrm{arg}(t_n-z)|<\frac{\pi}{2}. 
\end{align*}

The formula \eqref{eq_integral_representation_for_GHS} implies 
that the domain $D^{(0)}_{n+1}$ of integration describes the asymptotic 
behaviour of the holomorphic solution at $z=0$. 
There are domains of integration yielding 
the asymptotic 
behaviours of the non-holomorphic solutions at $z=0$ 
with the characteristic exponents $1-\beta_i$ ($1\le i\le n$), 
and 
the non-holomorphic solutions at $z=\infty$ 
with the characteristic exponents $1-\beta_i$ ($1\le i\le n+1$) 
\cite{Mimachi intersection numbers for n+1Fn}. 
In order to define integrals for domains of integration, 
we should fix branches of $u(t)$.

For the convenience, suppose $z \in \R-\{0,1\}$. 
We first fix branches of $u(t)$ for real $z$, and then we consider 
analytic continuation of $u(t)$ for general $z$. 
For a simply connected domain $D$ in the real part $T_\R$ of $T_z$,  
let us define $u_D(t)$ as
\begin{align*}
u_D(t)=\prod_{i=1}^n (\epsilon_it_i)^{\lambda_i} 
\prod_{i=1}^{n+1} (\eta_i(t_{i-1}-t_i))^{\mu_i}, 
\end{align*}
where $\epsilon_i,\eta_i\in\{1,-1\}$ such that $\epsilon_it_i>0$ 
and $\eta_i(t_{i-1}-t_i)>0$ on $D$. We fix the arguments of all $\epsilon_i t_i$ 
and $\eta_i(t_{i-1}-t_i)$ as $0$. The function $u_D(t)$ 
is a branch of $u(t)$ multiplied by a scalar.

In what follows, 
for convenience we use 
\begin{align*}
&e(A) =\exp (\pi \sqrt{-1} A), \quad
s(A) =\sin (\pi A) \quad (A \in \C), 
\\
&\alpha_{i,j}=\sum_{s=i}^j\alpha_s,\quad 
\beta_{i,j}=\sum_{s=i}^j\beta_s\quad (i<j), 
\end{align*}
and for the exponents $\lambda_i, \mu_i$ of $u(t)$  
\begin{align*}
\lambda_{i, j} &=
\begin{cases}
\lambda_i + \cdots +\lambda_j & (i \leq j) , \\
0 & (i=j+1) , \\
-(\lambda_{j+1}+ \cdots +\lambda_{i-1}) & (i \geq j+2) , 
\end{cases} \\
\mu_{i, j} &=
\begin{cases}
\mu_i + \cdots +\mu_j & (i \leq j) , \\
0 & (i=j+1) , \\
-(\mu_{j+1}+ \cdots +\mu_{i-1}) & (i \geq j+2) ,
\end{cases} \\
e_{i, j} &= e(\lambda_{i, j}), \\
\tilde{e}_{i, j} &= e(\mu_{i, j}).   
\end{align*}
We note that for all $i$ we have 
\begin{align*}
\lambda_{i}=\lambda_{i, i}, 
\quad \mu_{i}=\mu_{i, i}.
\end{align*}

\subsection{Connection problem between fundamental solutions at $z=0$ and $z=\infty$}

In this subsection, we recall the results in \cite{Mimachi intersection numbers for n+1Fn}. 
We fix $z \in \C$ such that $z < 0$. Set  
\begin{align*}
D_i^{(0)}=&\{t\in T_\R \mid z<t_n< \cdots <t_i <0,\ 1<t_1< \cdots < t_{i-1} \} \quad (1 \leq i \leq  n+1),
\\
D_i^{(\infty)}=&\{t\in T_\R \mid t_i< \cdots <t_n <z,\ 0<t_{i-1}< \cdots < t_1< 1 \} \quad (1 \leq i \leq  n+1).
\end{align*}
The domains of integration in the case of $n=2$ are pictured as follows.
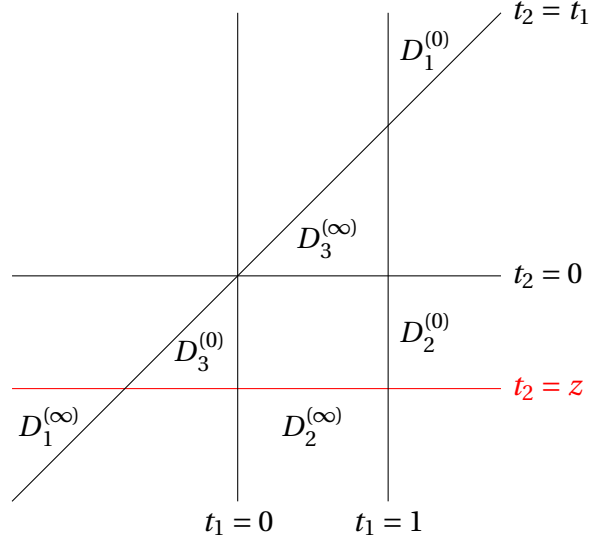
\begin{center}
\begin{tikzpicture}
\draw (-3,0)--(3.5,0)node[anchor=west]{$t_2=0$};
\draw [color=red](-3,-1.5)--(3.5,-1.5)node[anchor=west]{$t_2=z$};
\draw (0,3.5)--(0,-3)node[anchor=north]{$t_1=0$};
\draw (2,3.5)--(2,-3)node[anchor=north]{$t_1=1$};
\draw (-3,-3)--(3.5,3.5)node[anchor=west]{$t_2=t_1$};
\coordinate node at (2.5,3) {$D_1^{(0)}$};
\coordinate node at (2.5,-0.75) {$D_2^{(0)}$};
\coordinate node at (1.2,0.5) {$D_3^{(\infty)}$};
\coordinate node at (1,-2) {$D_2^{(\infty)}$};
\coordinate node at (-0.5,-1) {$D_3^{(0)}$};
\coordinate node at (-2.5,-2) {$D_1^{(\infty)}$};
\end{tikzpicture} 
\end{center}

The orientation of the domains $D_i^{(0)}$ or $D_i^{(\infty)}$ 
of integration is fixed 
to be natural one induced from $T_\R$. 
In \cite{Mimachi intersection numbers for n+1Fn}, 
it was shown that the domains $D_i^{(0)}$ 
of integration give a fundamental system of solutions at $z=0$, 
and $D_{i}^{(\infty)}$ gives a fundamental system of solutions at $z=\infty$. 
\begin{prop}
[\cite{Mimachi intersection numbers for n+1Fn}, Proposition 2.1]
\label{[M2], prop 2.1}
(1) For a fixed $i$ such that  $1 \leq i \leq n+1$,  
if ${\rm Re}(\alpha_i-\beta_s+1)>0$ and ${\rm Re}(\beta_s-\alpha_s)>0$
for $1 \leq s \leq n+1$ with $s \neq i$ and $|z|>1$, then we have 
\begin{align}
&\int _{D_i^{(\infty)}} u_{D_i^{(\infty)}}(t)dt_1 \cdots dt_n \label{sol_infty} 
=\prod_{1 \leq s \leq n+1, s \neq i} B(\alpha_i-\beta_s+1, \beta_s-\alpha_s)f_i^{(\infty)}(z),
\end{align}
where 
\begin{align*}
f_i^{(\infty)}(z)&=(-z)^{-\alpha_i}  {}_{n+1} F _n \left( \begin{matrix}
\alpha_i-\beta_1+1 ,  \alpha_i-\beta_2+1, \dots ,  \alpha_i-\beta_{n+1}+1 \\
\alpha_i-\alpha_1+1 ,  \dots ,  \widehat{\alpha_i-\alpha_i+1} ,  \dots ,  \alpha_i-\alpha_{n+1}+1 \\
\end{matrix} ; \frac{1}{z} \right). 
\end{align*}

(2) For a fixed $i$ such that  $1 \leq i \leq n+1$,  
if ${\rm Re}(\alpha_s-\beta_i+1)>0$, ${\rm Re}(\beta_s-\alpha_s)>0$ 
for $1 \leq s \leq n+1$ with $s \neq i$ and $|z|<1$, then we have 
\begin{align}
&\int _{D_i^{(0)}} u_{D_i^{(0)}}(t)dt_1 \cdots dt_n \label{sol_0} 
=\prod_{1 \leq s \leq n+1, s \neq i} B(\alpha_s-\beta_i+1, \beta_s-\alpha_s)f_i^{(0)}(z),
\end{align}
where 
\begin{align*}
f_i^{(0)}(z)&=(-z)^{1-\beta_i} {}_{n+1} F _n \left( \begin{matrix}
\alpha_1-\beta_i+1 ,  \alpha_2-\beta_i+1, \dots ,  \alpha_{n+1}-\beta_i+1 \\
\beta_1-\beta_i+1 ,  \dots ,  \widehat{\beta_i-\beta_i+1} ,  \dots ,  \beta_{n+1}-\beta_i+1 \\
\end{matrix} ; z \right). 
\end{align*}
\qed 
\end{prop}

Let $F_D(z)=\int_D u_D(t)dt$. 
\begin{prop}[\cite{Mimachi intersection numbers for n+1Fn}, Proposition 2.5, 
Theorem 2.6]
\label{[M2], prop 2.5}
For $i$ and $j$ such that $1\leq i,j\leq n+1$, suppose that 
\begin{align*}
&\mathrm{Re}(\alpha_i-\beta_j+1)>0,\quad \mathrm{Re}(\beta_j-\alpha_j)>0
\quad (i\neq j), 
\\
&\alpha_i-\beta_j\notin\Z,\quad \beta_i-\beta_j\notin\Z\quad (i\neq j). 
\end{align*}
Then we have 
\begin{equation}\label{eq_connection_formula_i0}
F_{D_i^{(\infty)}}(z)=\sum_{1 \leq j \leq n+1} \frac{s(\beta_i-\alpha_i)}{s(\beta_j-\alpha_i)} \prod_{1 \leq s \leq n+1,\atop s \neq j} \frac{s(\alpha_s-\beta_j)}{s(\beta_s-\beta_j)} \times F_{D_j^{(0)}}(z) 
\end{equation}
for $1\leq i \leq n+1$. \qed 
\end{prop}

\section{Connection problem between fundamental system of solutions at $z=0$ and $z=1$}

\subsection{Asymptotic behaviour}
We fix $z \in \C$ such that $0<z < 1$. Set
\begin{align}
\tilde{D}_i^{(0)}&=\{t\in T_\R \mid 0<t_i< \cdots <t_n <z,\ 1<t_1< \cdots < t_{i-1} \} \quad (1 \leq i \leq  n+1), \label{domain0_0<z<1} 
\\
\tilde{D}_i^{(1)}&=\{t\in T_\R | t_i< \cdots <t_n <0,\ 0<t_{i-1}< \cdots < t_1< 1 \} \quad (1 \leq i \leq  n), \label{holo_domain1_0<z<1}
 \\
\tilde{D}_{n+1}^{(1)}&=\{t\in T_\R | z<t_n< \cdots < t_1< 1 \}. \label{nonholo_domain1_0<z<1}
\end{align}
The domains of integration in the case of $n=2$ are pictured as follows. 
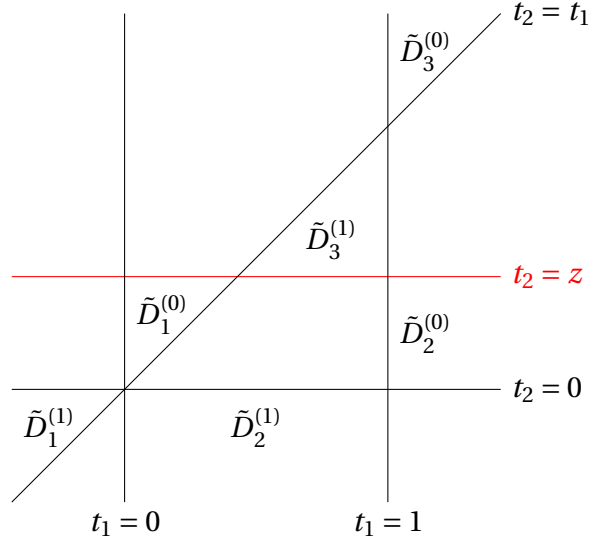
\begin{center}
\begin{tikzpicture}
\draw [color=red](-3,0)--(3.5,0)node[anchor=west]{$t_2=z$};
\draw (-3,-1.5)--(3.5,-1.5)node[anchor=west]{$t_2=0$};
\draw (-1.5,3.5)--(-1.5,-3)node[anchor=north]{$t_1=0$};
\draw (2,3.5)--(2,-3)node[anchor=north]{$t_1=1$};
\draw (-3,-3)--(3.5,3.5)node[anchor=west]{$t_2=t_1$};
\coordinate node at (2.5,3) {$\tilde{D}_3^{(0)}$};
\coordinate node at (2.5,-0.75) {$\tilde{D}_2^{(0)}$};
\coordinate node at (1.25,0.5) {$\tilde{D}_3^{(1)}$};
\coordinate node at (-1,-0.5) {$\tilde{D}_1^{(0)}$};
\coordinate node at (0.25,-2) {$\tilde{D}_2^{(1)}$};
\coordinate node at (-2.5,-2) {$\tilde{D}_{1}^{(1)}$};
\end{tikzpicture} 
\end{center}

The orientation of the domains $\tilde{D}_i^{(0)}$ or $\tilde{D}_i^{(1)}$ 
of integration is fixed 
to be natural one induced from $T_\R$. In \cite{Mimachi intersection numbers for n+1Fn}, 
it was shown that the domains $\tilde{D}_i^{(0)}$ 
of integration give a fundamental system of solutions at $z=0$, 
and $\tilde{D}_{n+1}^{(1)}$ gives the non-holomorhic solution at $z=1$. 
\begin{prop}[\cite{Mimachi intersection numbers for n+1Fn}, 
Proposition 3.1]
\label{[M2], prop 3.1}
(1) For a fixed $i$ such that  $1 \leq i \leq n+1$,  
if ${\rm Re}(\alpha_s-\beta_i+1)>0$ and ${\rm Re}(\beta_s-\alpha_s)>0$
for $1 \leq s \leq n+1$ with $s \neq i$ and $|z|<1$, then we have 
\begin{align*}
&\int _{\tilde{D}_i^{(0)}} u_{\tilde{D}_i^{(0)}}(t)dt_1 \cdots dt_n
=\prod_{1 \leq s \leq n+1,\atop s \neq i} B(\alpha_s-\beta_i+1, \beta_s-\alpha_s)f_i^{(0)}(z),
\end{align*}
where 
\begin{align*}
f_i^{(0)}(z)&=z^{1-\beta_i} {}_{n+1} F _n \left( \begin{matrix}
\alpha_1-\beta_i+1 ,  \alpha_2-\beta_i+1, \dots ,  \alpha_{n+1}-\beta_i+1 \\
\beta_1-\beta_i+1 ,  \dots ,  \widehat{\beta_i-\beta_i+1} ,  \dots ,  \beta_{n+1}-\beta_i+1 \\
\end{matrix} ; z \right).  
\end{align*}

(2)  
If ${\rm Re}(\beta_{1,s}-\alpha_{1,s})>0$ 
for $1 \leq s \leq n$ and 
${\rm Re}(\beta_s-\alpha_s)>0$
for $1 \leq s \leq n+1$, and $|1-z|<1$, then we have
\begin{align}
&\int _{\tilde{D}_{n+1}^{(1)}} u_{\tilde{D}_{n+1}^{(1)}}(t)dt_1 \cdots dt_n \label{nonholosol_1}
=\prod_{s=1}^{n} B(\beta_{1,s}-\alpha_{1,s}, \beta_{s+1}-\alpha_{s+1})f_{n+1}^{(1)}(z), \end{align}
where
\begin{align*}
f_{n+1}^{(1)}(z)=&(1-z)^{\beta_{1,n}-\alpha_{1,n+1}}  \sum_{i_1, \dots, i_n \geq 0} 
\prod_{s=1}^{n} \frac{(\beta_s-\alpha_{s+1})}{i_s!}\prod_{s=1}^{n}\frac{(\sum_{k=1}^{s} (\beta_k-\alpha_k))_{i_1+\cdots+i_s}}
{(\sum_{k=1}^{s+1}(\beta_k-\alpha_k))_{i1+\cdots+i_s}}(1-z)^{i_1+\cdots+i_n}.
\end{align*}
\qed 
\end{prop}

\begin{prop}\label{holo_sols_at_1}
For a fixed $i$ such that  $1 \leq i \leq n$,  if 
\begin{align*}
&\mathrm{Re}(\alpha_i-\beta_s+1)>0\quad (1\leq s\leq n+1,\ s\neq i),
\\
&\mathrm{Re}(\beta_s-\alpha_s)>0\quad (1\leq s\leq n,\ s\neq i),
\\
&\mathrm{Re}(\alpha_{n+1}-\beta_i+1)>0, 
\end{align*}
 and $|1-z|<1$, then we have 
\begin{align}
\label{eq_holosol_1}
\int _{\tilde{D}_{i}^{(1)}} u_{\tilde{D}_{i}^{(1)}}(t)dt_1 \cdots dt_n =&\prod_{s=1,s\neq i}^{n} B(\alpha_i-\beta_s+1, \beta_s-\alpha_s) 
  B(\alpha_i, \alpha_{n+1}-\beta_i+1) f_{i}^{(1)}(z), 
\end{align}
where 
\begin{align*}
f_{i}^{(1)}(z)=&\sum_{m_1, m_2 \geq 0}\frac{(\alpha_{n+1})_{m_1}
(\alpha_i-\beta_i+1)_{m_2}
 (\alpha_1)_{m_1+m_2}}{m_1!(\alpha_i+\alpha_{n+1}-\beta_i+1)_{m_1+m_2}} \sum_{m_3=0}^{m_2}
 \frac{(-1)^{m_3}}{ m_3 ! (m_2-m_3)!}
 \prod_{s=1,s\neq i}^{n}\frac{(\alpha_i-\beta_s+1)_{m_3}}
{(\alpha_i-\alpha_s+1)_{m_3}}
 (1-z)^{m_1}.
\end{align*}
\qed 
\end{prop}
\begin{proof}
We change the integration variables as
\begin{align*}
t_s&=u_1u_2 \cdots u_s \ (1 \leq s \leq i-1) , \\
t_s&=u_s^{-1}u_{s+1}^{-1} \cdots u_n^{-1}(u_n-1) \ (i \leq s \leq n).
\end{align*}
Its Jacobian is 
\begin{align*}
\frac{\partial(t_1, \dots , t_n)}{\partial(u_1, \dots, u_n)}=u_1^{i-2}u_2^{i-3}\cdots u_{i-2}^{1}u_{i-1}^{0}u_{i}^{-2}u_{i+1}^{-3}\cdots u_{n-1}^{i-n-1}u_n^{i-n-2}(1-u_n)^{n-i}. 
\end{align*}
Hence, we have 
\begin{align*}
&\int _{\tilde{D}_{i}^{(1)}} u_{\tilde{D}_{i}^{(1)}}(t)dt_1 \cdots dt_n
\\
=&\int_{(0, 1)^n} \prod_{s=1}^{i-1}u_s^{\lambda_{s, i-1}+\mu_{s+1, i-1}+i-s-1}\prod_{s=i}^{n}u_s^{-\lambda_{i, s}-\mu_{i, s+1}+i-s-2}\prod_{s=1}^{i-1}(1-u_s)^{\mu_s}  \prod_{s=i}^{n-1}(1-u_s)^{\mu_{s+1}} (1-u_n)^{\lambda_{i, n}+\mu_{i+1, n}+n-i}
 \\
& \times \left( 1-u_n(1-u_1\cdots u_{n-1}) \right)^{\mu_j} \left( 1-(1-z)u_n \right)^{\mu_{n+1}}du_1\cdots du_n
\\
=&\sum_{m_1, m_2 \geq 0} \sum_{m_3=0}^{m_2} \frac{(-\mu_{n+1})_{m_1}(-\mu_i)_{m_2}}{m_1 ! m_3 ! (m_2-m_3)!}(1-z)^{m_1}  \prod_{s=1,s\neq i}^{n} \int_{0}^{1} u_s^{\lambda_{s, i-1}+\mu_{s+1, i-1}+i-s+m_3-1}(1-u_s)^{\mu_s}du_s
 \\
&\times \int_{0}^{1} u_n^{-\lambda_{i, n}-\mu_{i, n+1}+i-n+m_1+m_2-2}(1-u_n)^{\lambda_{i, n}+\mu_{i+1, n}+n-i}du_n
\end{align*}
 by the binomial theorem
 \begin{align*}
(1-(1-z)u_n)^{\mu_{n+1}}&= \sum_{m_1 \geq 0} \frac{(-\mu_{n+1})_{m_1}}{m_1 !} u_n^{m_1}(1-z)^{m_1} ,
 \\
(1-u_n(1-u_1\cdots u_{n-1}) )^{\mu_i}
&=\sum_{m_2 \geq 0} \sum_{m_3=0}^{m_2} \frac{(-\mu_i)_{m_2}(-1)^{m_3}}{m_3 !(m_2-m_3)!}u_1^{m_3}\cdots u_{n-1}^{m_3}u_n^{m_2}.  
\end{align*}
Using the formula
\begin{align*}
B(\alpha+m,\beta)=\frac{(\alpha)_m}{(\alpha+\beta)_m}B(\alpha,\beta), 
\end{align*}
we obtain 
 \begin{align*}
 &\int _{\tilde{D}_{i}^{(1)}} u_{\tilde{D}_{i}^{(1)}}(t)dt_1 \cdots dt_n
\\
=&\sum_{n_1, n_2 \geq 0} \sum_{n_3=0}^{n_2} \frac{(-\mu_{n+1})_{n_1}(-\mu_i)_{n_2}}{n_1 ! n_3 ! (n_2-n_3)!}(1-z)^{n_1}  \prod_{s=1,s\neq i}^{n} B(\lambda_{s, i-1}+\mu_{s+1, i-1}+i-s+n_3, \mu_s +1) \\
&\times B(-\lambda_{i, n}-\mu_{i, n+1}+i-n+n_1+n_2-1, \lambda_{i, n}+\mu_{i+1, n}+n-i+1) \\
=&\prod_{s=1,s\neq i}^{n} B(\lambda_{s, i-1}+\mu_{s+1, i-1}+i-s, \mu_s +1)  B(-\lambda_{i, n}-\mu_{i, n+1}+i-n-1, \lambda_{i, n}+\mu_{i+1, n}+n-i+1) \\
&\times \sum_{n_1, n_2 \geq 0} \sum_{n_3=0}^{n_2} \frac{(-\mu_{n+1})_{n_1}(-\mu_i)_{n_2}(-\lambda_{i, n}-\mu_{i, n+1}-n-1)_{n_1+n_2}}{n_1 !n_3 ! (n_2-n_3)!(-\mu_j-\mu_{n+1})_{n_1+n_2}} 
 \prod_{s=1,s\neq i}^{n}\frac{(\lambda_{s, i-1}+\mu_{s+1, i-1}+i-s)_{n_3}}{(\lambda_{s, i-1}+\mu_{s, i-1}+i-s+1)_{n_3}}
 (1-z)^{n_1}.
\end{align*}
Finally the transformation \eqref{eq_transformation_lambda_alpha_beta} yields
 \eqref{eq_holosol_1}.  
\end{proof}

\subsection{Main theorem}

In this subsection, we solve connection problem between fundamental  systems of solutions at 
singular points $0$ and $1$ for the generalized hypergeometric equation.

\begin{thm}\label{thm_connection_formula_10}
Assume for $1 \leq i,  j \leq n+1$
\begin{align*}
&\mathrm{Re}(\alpha_i-\beta_j+1)>0\quad (j\neq i),
\\
&\mathrm{Re}(\beta_j-\alpha_j)>0, 
\\
&\alpha_i-\beta_j \notin \Z, \quad \beta_i-\beta_j \notin \Z \quad (j\neq i). 
\end{align*}
Then   we have 
\begin{equation}\label{eq_connection_formula_10}
F_{\tilde{D}_i^{(1)}}(z)=\sum_{1\le j\le n+1}\frac{s(\beta_i-\alpha_i)s(\alpha_{n+1})}
{s(\beta_j-\alpha_i)s(\beta_j-\alpha_{n+1})}
\prod_{1\le k\le n+1,\atop k\neq j}\frac{s(\alpha_k-\beta_j)}{s(\beta_k-\beta_j)}
\times 
F_{\tilde{D}_j^{(0)}}(z)
\end{equation}
for $1\le i\le n$ and  
\begin{equation}\label{eq_connection_formula_10_n+1}
F_{\tilde{D}_{n+1}^{(1)}}(z)=
\sum_{j=1}^{n+1} \prod_{1 \leq k \leq n+1, \atop k \neq j} 
\frac{s(\alpha_k-\beta_j)}{s(\beta_k-\beta_j)} \times 
F_{\tilde{D}_j^{(0)}}(z).
\end{equation}
\qed
\end{thm}
\begin{proof}
By induction. If $n=1$, the connection formulas \eqref{eq_connection_formula_10} 
and \eqref{eq_connection_formula_10_n+1}
are for the Gauss hypergeometric function. 

Suppose the identity \eqref{eq_connection_formula_10} is true for $k<n$. 
Our strategy to prove the theorem is to calculate analytic continuation 
of the connection formula \eqref{eq_connection_formula_i0} along the following path $\gamma$ on the plane of the variable $z$. 
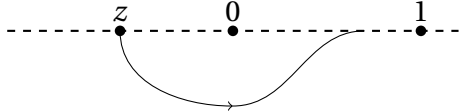
\begin{center}
\begin{tikzpicture}
\draw [thick,dashed] (-3,0)--(3,0);
\draw [->] (-1.5,0) to [out=270, in=180] (0,-1);
\draw (0,-1) to [out=0, in=180] (1.75,0);
\fill (0, 0)node[anchor=south]{$0$} circle [radius=2pt];
\fill (2.5, 0)node[anchor=south]{$1$} circle [radius=2pt];
\fill (-1.5, 0)node[anchor=south]{$z$} circle [radius=2pt];
\end{tikzpicture}
\end{center}
First we compute analytic continuation of the integral representations 
 along the path $\gamma$. In what follows, we abbreviate $F_D(z)$ as $D$, 
 and $\{t\in T_\R \mid \cdots\}$ as  $\{\cdots\}$. 
\begin{lem}\label{lem_analytic_continuation_of_domain}
By analytic continuation along the path $\gamma$, we have 
\begin{align}
 \gamma^{*}\left(D_j^{(0)}\right)=&(-1)^{n-j+1}e_{j, n}\tilde{e}_{j+1, n+1}\tilde{D}_j^{(0)}\quad (1 \leq j \leq n+1),
 \\
\label{eq_analytic_continuation_infinity} 
\gamma^{*}\left(D_j^{(\infty)}\right)=&\tilde{D}_{j}^{(1)}
 +\sum_{k=j+1}^{n}e_{k, n} \biggl\{ \begin{array}{c} 
t_j< \cdots <t_{k-1} <0, 
\\
0<t_k< \cdots < t_n<z,\ 0<t_{j-1}< \cdots <t_1<1
\end{array} \biggr\}
\\
&+e_{j, n} \{ 0<t_j< \cdots <t_n <z,\ 0<t_j< \cdots < t_1<1\}\nonumber
\\
&+e_{j, n} \tilde{e}_{j} \{ 0<t_{j-1}<t_j< \cdots <t_n <z, \ 
0<t_{j-1}< \cdots < t_1< 1\}\quad (1 \leq j \leq n), 
\nonumber
\\
\gamma^{*}\left(D_{n+1}^{(\infty)}\right)=&\tilde{D}_{n+1}^{(1)}
+\tilde{e}_{n+1}\{ t_n<z,\ 0<t_n<\cdots <t_1<1\}, 
\label{eq_analytic_continuation_infinity_n+1}
\end{align}
where $\gamma^{*}(D)$ stands for the result of 
analytic continuation along the path 
$\gamma$ for $D$. 
\qed 
\end{lem}
\begin{proof}
By the definition of $D_j^{(0)}$, analytic continuation 
along the path 
$\gamma$ increases the arguments of $t_k$  ($j \leq k \leq n$), 
$t_{k-1}-t_k$  ($j+1 \leq k \leq n+1)$ by $\pi$. Hence, we have 
\begin{align*}
\gamma^{*}(D_j^{(0)})=e_{j, n}\tilde{e}_{j+1, n+1} \{ \overleftarrow{0< t_j< \cdots <t_n <z},\ 1<t_1< \cdots < t_{j-1}\}. 
\end{align*}
Here, $ \overleftarrow{0< t_j< \cdots <t_n <z}$ 
means that the orientation is inverse. Hence we obtain
\begin{align*}
\gamma^{*}(D_j^{(0)})&=e_{j, n}\tilde{e}_{j+1, n+1} \{ \overleftarrow{0< t_j< \cdots <t_n <z},\
 1<t_1< \cdots < t_{j-1}\} \\
&=(-1)^{n-j+1}e_{j, n}\tilde{e}_{j+1, n+1} \{ 0< t_j< \cdots <t_n <z,\ 1<t_1< \cdots < t_{j-1} \}. 
\end{align*}

Next we compute analytic continuation  of 
 $D_j^{(\infty)}$. By the definition, 
 after analytic continuation along the path 
$\gamma$, 
 the line $\{ t_n <z\}$ 
divides into  
two lines  $\{ t_n <0\}$ and $\{0<t_n<z\}$ 
on the plane of the variable $t_n$.  
Since the argument of $t_n$ increases $\pi$ on $\{0<t_n<z\}$, 
the line $\{0<t_n<z\}$ is multiplied by $e_n$. 
Similarly, $\{t_{n-1}<t_n,\ 0<t_n<z\}$ divides into two lines 
$\{t_{n-1}<0\}$ and $\{0<t_{n-1}<t_n\}$
on the plane of the variable $t_{n-1}$.  
Since the argument of $t_{n-1}$ increases $\pi$ on $\{0<t_{n-1}<t_n\}$, 
the line $\{0<t_{n-1}<t_n\}$ is multiplied by $e_{n-1}$. By repeating 
this procedure  from $t_{n-2}$ plane to $t_j$ plane, we have
\begin{align*}
\gamma^{*}(D_j^{(\infty)})=&\tilde{D}_{j}^{(1)}
+\sum_{k=j+1}^{n}e_{k, n} \biggl\{
    \begin{array}{c}
      t_j< \cdots <t_{k-1} <0,  \\
      0<t_k< \cdots < t_n<z,\ 0<t_{j-1}< \cdots <t_1<1
    \end{array}
  \biggr\} \\
&+e_{j, n} \left\{ 0<t_{j}< \cdots <t_n <z,\  
0<t_{j-1}< \cdots < t_1<1\right\}.
\end{align*}
The last domain  is actually divided as 
\begin{align*}
 \{ 0<t_j< \cdots <t_n <z,\ 0<t_{j-1}< \cdots < t_1<1 \} 
=&\tilde{e}_{j} \{ 0<t_{j-1}<t_j< \cdots <t_n <z,\ 0<t_{j-1}<t_{j-2}< \cdots < t_1<1 \} \\
&+ \{ 0<t_j< \cdots <t_n <z,\ t_j<t_{j-1}< \cdots < t_1<1 \},  
\end{align*}
because the line $\{0<t_{j-1}<t_{j-2}\}$ splits into two lines
$\{0<t_{j-1}<t_j\}$ and $\{0<t_j<t_{j-1}<t_{j-2}\}$ on the 
plane of the variable $t_{j-1}$, and 
the argument of $(t_{j-1}-t_j)$ increases $\pi$  on $\{0<t_{j-1}<t_j\}$. 
Therefore, we obtain \eqref{eq_analytic_continuation_infinity}. 
In the same way, we obtain \eqref{eq_analytic_continuation_infinity_n+1}. 
\end{proof}

We back to the proof of the theorem. 
The connection formula in Proposition \ref{[M2], prop 2.5} 
reads as 
\begin{align*}
D_j^{(\infty)}=\sum_{ k =1}^{n+1} (-1)^{n+k-j} \frac{s(\mu_j)}{s(\lambda_{j, k-1}+\mu_{j, k})} \prod_{1 \leq l \leq n+1,\atop l \neq k} \frac{s(\lambda_{l, k-1}+\mu_{l, k})}{s(\lambda_{l, k-1}+\mu_{l+1, k})}  D_k^{(0)}
\end{align*}
for the parameters $\lambda_i$, $\mu_i$ by the relations 
\eqref{eq_transformation_lambda_alpha_beta}. 
Due to Lemma \ref{lem_analytic_continuation_of_domain},  
 analytic continuation of this for $j=1,\ldots,n$
  along the path $\gamma$ is  
 equal to
\begin{align}
\label{eq_annalytic_continuation_connection_formula_1}
\tilde{D}_{j}^{(1)}
=&\sum_{ k=1}^{ n+1} (-1)^{n+k-j} \frac{s(\mu_j)}{s(\lambda_{j, k-1}+\mu_{j, k})} \prod_{1 \leq l \leq n+1,\atop l \neq k} 
\frac{s(\lambda_{l, k-1}+\mu_{l, k})}{s(\lambda_{l, k-1}+\mu_{l+1, k})} (-1)^{n-k+1}e_{k, n}\tilde{e}_{k+1, n+1}\tilde{D}_k^{(0)}\nn
\\
&-\sum_{k=j+1}^{n}e_{k, n} \{ -\infty<t_j< \cdots <t_{k-1} <0,\ 0<t_k< \cdots < t_n<z,\ 0<t_{j-1}< \cdots <t_1<1\}\nonumber
\\
&-e_{j, n} \{ 0<t_j< \cdots <t_n <z,\ 0<t_j< \cdots < t_1<1\}\\
&-e_{j, n} \tilde{e}_{j} \{ 0<t_{j-1}<t_j< \cdots <t_n <z,\ 
 0<t_{j-1}< \cdots < t_1< 1\}. \nonumber
\end{align}
In the integral associated with 
$\{ t_j< \cdots <t_{k-1} <0,\ 0<t_k< \cdots < t_n<z,\ 0<t_{j-1}< \cdots <t_1<1\}$ 
in the second term of \eqref{eq_annalytic_continuation_connection_formula_1},  
the integral with respect to the variables 
$t_1,\ldots, t_{k-1}$ corresponds to the integral defined for
$\tilde{D}_{j}^{(1)}$ of ${}_kF_{k-1}(t_k)$ because of $0<t_k<1$. 
Hence, we can apply the induction hypothesis 
to $\{ t_j< \cdots <t_{k-1} <0,\ 0<t_k< \cdots < t_n<z,\ 0<t_{j-1}< \cdots <t_1<1\}$ and obtain a sum of integrals 
associated with 
$\{ 0<t_k< \cdots < t_n<z,\ 0<t_m<\cdots<t_k,\ 1<t_1<\cdots<t_{m-1}\}$ 
($m=1,\ldots,k$), 
which is nothing but $\tilde{D}_m^{(0)}$ of ${}_{n+1}F_n(z)$. 
In the same way, we can regard the third term 
$\{ 0<t_j< \cdots <t_n <z,\ 0<t_j< \cdots < t_1<1\}$ 
in \eqref{eq_annalytic_continuation_connection_formula_1} 
as 
$\tilde{D}_j^{(1)}$ of ${}_jF_{j-1}(t_j)$. 
Hence, we can apply the induction hypothesis 
to $\{ 0<t_j< \cdots <t_n <z,\ 0<t_j< \cdots < t_1<1\}$
and obtain a sum of integrals 
associated with 
$\{ 0<t_j< \cdots < t_n<z,\ 0<t_k<\cdots<t_j,\ 1<t_1<\cdots<t_{k-1}\}$ 
($k=1,\ldots,j$), 
which is nothing but $\tilde{D}_k^{(0)}$ of ${}_{n+1}F_n(z)$. 
We can also regard the fourth term $\{ 0<t_{j-1}<t_j< \cdots <t_n <z,\ 
 0<t_{j-1}< \cdots < t_1< 1\}$  in \eqref{eq_annalytic_continuation_connection_formula_1}  as 
 $\tilde{D}_{j-1}^{(1)}$ of ${}_{j-1}F_{j-2}(t_{j-1})$. 
Hence, we can apply the induction hypothesis 
to $\{ 0<t_{j-1}<t_j< \cdots <t_n <z,\ 
 0<t_{j-1}< \cdots < t_1< 1\}$
and obtain a sum of integrals 
associated with 
$\{ 0<t_{j-1}< \cdots < t_n<z,\ 0<t_k<\cdots<t_{j-1},\ 
1<t_1<\cdots<t_{k-1}\}$ 
($k=1,\ldots,j-1$), 
which is nothing but $\tilde{D}_k^{(0)}$ of ${}_{n+1}F_n(z)$.

Since the identities \eqref{eq_connection_formula_10} and 
\eqref{eq_connection_formula_10_n+1} read as
\begin{align*}
\tilde{D}_j^{(1)}=&(-1)^{j-1} \sum_{k=1}^{n}
\frac{s(\mu_j)s(\mu_{n+1})}{s(\lambda_{j, k-1}+\mu_{j, k})
s(\lambda_{n+1, k-1}+\mu_{n+2, k})} 
\prod_{l=1, l\neq k}^{n}\frac{s(\lambda_{l, k-1}+\mu_{l, k})}
{s(\lambda_{l, k-1}+\mu_{l+1, k})}   \tilde{D}_{k}^{(0)} 
\\
&+(-1)^{j-1}\frac{s(\mu_j)}{s(\lambda_{j, n}+\mu_{j, n+1})} 
\prod_{l=1}^{n}\frac{s(\lambda_{l, n}+\mu_{l, n+1})}
{s(\lambda_{l, n}+\mu_{l+1, n+1})} \tilde{D}_{n+1}^{(0)} 
\end{align*}
for the parameters $\lambda_i$, $\mu_i$ by the relations 
\eqref{eq_transformation_lambda_alpha_beta}, 
 for $j=1,\ldots, n$ we obtain  
\begin{align}
\label{eq_connection_formula_proof_j=1-n}
\tilde{D}_{j}^{(1)} 
=&\sum_{ k=1}^{  n+1} (-1)^{-j+1} e_{k, n}\tilde{e}_{k+1, n+1} \frac{s(\mu_j)}{s(\lambda_{j, k-1}+\mu_{j, k})} \prod_{1 \leq l \leq n+1, l \neq k} \frac{s(\lambda_{l, k-1}+\mu_{l, k})}
{s(\lambda_{l, k-1}+\mu_{l+1, k})}  \tilde{D}_k^{(0)} 
\\
&-\sum_{k=j+1}^{n} e_{k, n} \left\{ (-1)^{j-1} \sum_{m=1}^{k-1}
\frac{s(\mu_j)s(\mu_{k})}{s(\lambda_{j, m-1}+\mu_{j, m})
s(\lambda_{k, m-1}+\mu_{k+1, m})} 
\prod_{l=1, l\neq m}^{k-1}\frac{s(\lambda_{l, m-1}+\mu_{l, m})}{s(\lambda_{l, m-1}+\mu_{l+1, m})}   \tilde{D}_{m}^{(0)}
\right. \nn
\\
&\left.+(-1)^{j-1}\frac{s(\mu_j)}{s(\lambda_{j, k-1}+\mu_{j, k})} 
\prod_{l=1}^{k-1}\frac{s(\lambda_{l, k-1}+\mu_{l, k})}
{s(\lambda_{l, k-1}+\mu_{l+1, k})} \tilde{D}_{k}^{(0)} 
\right\}\nn
 \\
&-e_{j, n}\sum_{k=1}^{j}(-1)^{j-1}
 \prod_{l=1, l \neq k}^{j} \frac{s(\lambda_{l, k-1}+\mu_{l, k})}{s(\lambda_{l, k-1}+\mu_{l+1, k})}\tilde{D}_{k}^{(0)}
-e_{j, n}\tilde{e}_{j}\sum_{k=1}^{j-1}(-1)^{j} \prod_{l=1, l \neq k}^{j-1} \frac{s(\lambda_{l, k-1}+\mu_{l, k})}
{s(\lambda_{l, k-1}+\mu_{l+1, k})}\tilde{D}_{k}^{(0)} \nn
\\
=&\sum_{k=1}^{j-1} \left(A_k-\sum_{m=j+1}^{n}B_{k, m}-E_k-F_k \right) \tilde{D}_{k}^{(0)}+\sum_{k=j}^{n} \left( A_k-\sum_{m=k+1}^{n+1}B_{k, m}-C_k \right) \tilde{D}_{k}^{(0)},\nn
\end{align}
where 
\begin{align*}
A_k&=(-1)^{j-1}e_{k, n}\tilde{e}_{k+1, n+1}\frac{s(\mu_j)}{s(\lambda_{j, k-1}+\mu_{j, k})}\prod_{l=1, l \neq k}^{n+1} \frac{s(\lambda_{l, k-1}+\mu_{l, k})}{s(\lambda_{l, k-1}+\mu_{l+1, k})}\quad (1\leq k\leq n+1) , \\
B_{k, m}&=e_{m, n} (-1)^{j-1} \frac{s(\mu_j)s(\mu_{m})}{s(\lambda_{j, k-1}+\mu_{j, k})
s(\lambda_{m,k-1}+\mu_{m+1,k})}
\prod_{l=1, l \neq k}^{m-1} \frac{s(\lambda_{l, k-1}+\mu_{l, k})}{s(\lambda_{l, k-1}+\mu_{l+1, k})} \quad (1\leq m\leq n,\ 1\leq k\leq m), 
 \\
C_k&=e_{k, n}(-1)^{j-1}\frac{s(\mu_j)}{s(\lambda_{j, k-1}+\mu_{j, k})}
 \prod_{l=1}^{k-1}\frac{s(\lambda_{l, k-1}+\mu_{l, k})}{s(\lambda_{l, k-1}+\mu_{l+1, k})}
\quad (j\leq k \leq n+1) 
 , \\
E_k&=e_{j, n}(-1)^{j-1}\prod_{l=1, l \neq k}^{j} \frac{s(\lambda_{l, k-1}+\mu_{l, k})}{s(\lambda_{l, k-1}+\mu_{l+1, k})}
\quad (1\leq k\leq j-1) , \\
F_k&=e_{j, n}\tilde{e}_{j}(-1)^{j} \prod_{l=1, l \neq k}^{j-1} \frac{s(\lambda_{l, k-1}+\mu_{l, k})}{s(\lambda_{l, k-1}+\mu_{l+1, k})}
\quad (1\leq k\leq j-1).
\end{align*}

First, we compute the first term in the right hand side 
of \eqref{eq_connection_formula_proof_j=1-n} as follows. 
For $k=1,\ldots, j-1$ we have 
\begin{align*}
A_k-\sum_{m=j+1}^{n}B_{k, m}-E_k-F_k=&(-1)^{j-1}
\prod_{l=1, l \neq k}^{j-1} \frac{s(\lambda_{l, k-1}+\mu_{l, k})}{s(\lambda_{l, k-1}+\mu_{l+1, k})}
\\
&\times
\left(
e_{k, n}\tilde{e}_{k+1, n+1}\frac{s(\mu_j)}{s(\lambda_{j, k-1}+\mu_{j, k})}\prod_{l=j}^{n+1} \frac{s(\lambda_{l, k-1}+\mu_{l, k})}{s(\lambda_{l, k-1}+\mu_{l+1, k})}
\right.
\\
&-\sum_{m=j+1}^ne_{m, n}  \frac{s(\mu_j)s(\mu_{m})}{s(\lambda_{j, k-1}+\mu_{j, k})
s(\lambda_{m,k-1}+\mu_{m,k})}
\prod_{l=j}^{m-1} \frac{s(\lambda_{l, k-1}+\mu_{l, k})}{s(\lambda_{l, k-1}+\mu_{l+1, k})}
\\
&\left.-e_{j, n} \frac{s(\lambda_{j, k-1}+\mu_{j, k})}
{s(\lambda_{j, k-1}+\mu_{j+1, k})}+e_{j, n}\tilde{e}_{j}
\right). 
\end{align*}
The sum of the last term and the second last term above is computed as 
\begin{align*}
&-e_{j, n} \frac{s(\lambda_{k, j-1}+\mu_{k+1, j-1})}{s(\lambda_{k, j-1}+\mu_{k+1, j})}+e_{j, n}\tilde{e}_{j} 
\\
&=\frac{1}{s(\lambda_{k, j-1}+\mu_{k+1, j})} \frac{-e_{j, n}(e_{k, j-1}\tilde{e}_{k+1, j-1}-e_{k, j-1}^{-1}\tilde{e}_{k+1, j-1}^{-1})+e_{j, n}\tilde{e}_{j}(e_{k, j-1}\tilde{e}_{k+1, j}-e_{k, j-1}^{-1}\tilde{e}_{k+1, j}^{-1})}{2\sqrt{-1}}
\\
&=e_{k, n}\tilde{e}_{k+1, j}\frac{s(\mu_j)}
{s(\lambda_{k, j-1}+\mu_{k+1, j})}.
\end{align*}

Suppose for $i\geq j$ that it holds 
\begin{align*}
&-\sum_{m=j+1}^ie_{m,n}
\frac{s(\mu_j)s(\mu_{m})}{s(\lambda_{j, k-1}+\mu_{j, k})
s(\lambda_{m,k-1}+\mu_{m+1,k})}
\prod_{l=j}^{m-1} \frac{s(\lambda_{l, k-1}+\mu_{l, k})}{s(\lambda_{l, k-1}+\mu_{l+1, k})}+e_{k, n}\tilde{e}_{k+1, j}\frac{s(\mu_j)}
{s(\lambda_{k, j-1}+\mu_{k+1, j})}
\\
&=e_{k, n}\tilde{e}_{k+1, i}\frac{s(\mu_j)}
{s(\lambda_{k,j-1}+\mu_{k+1,j-1})}\prod_{l=j}^{i}
\frac{ s(\lambda_{k, l-1}+\mu_{k+1, l-1})}
 {s(\lambda_{k, l-1}+\mu_{k+1, l})}.
\end{align*}
Then we have 
\begin{align*}
&-\sum_{m=j+1}^{i+1}e_{m,n}
\frac{s(\mu_j)s(\mu_{m})}{s(\lambda_{j, k-1}+\mu_{j, k})
s(\lambda_{m,k-1}+\mu_{m+1,k})}
\prod_{l=j}^{m-1} \frac{s(\lambda_{l, k-1}+\mu_{l, k})}{s(\lambda_{l, k-1}+\mu_{l+1, k})}+e_{k, n}\tilde{e}_{k+1, j}\frac{s(\mu_j)}
{s(\lambda_{k, j-1}+\mu_{k+1, j})}
\\
&=\frac{s(\mu_j)}{s(\lambda_{k, j-1}+\mu_{k+1, j-1})}
\frac{\prod_{l=j}^{i} s(\lambda_{k, l-1}+\mu_{k+1, l-1})}
 {\prod_{l=j}^{i+1}s(\lambda_{k, l-1}+\mu_{k+1, l})}
\\
&\times  \frac{-e_{i+1,n}\left(\tilde{e}_{i+1}-\tilde{e}_{i+1}^{-1}\right)
+e_{k, n}\tilde{e}_{k+1, i}\left(e_{k,i}\tilde{e}_{k+1,i+1}
-e_{k,i}^{-1}\tilde{e}_{k+1,i+1}^{-1}\right)
}
{2\sqrt{-1}} 
 \\
&=e_{k, n}\tilde{e}_{k+1, i+1}\frac{s(\mu_j)}
{s(\lambda_{k,j-1}+\mu_{k+1,j-1})}\prod_{l=j}^{i+1}
\frac{ s(\lambda_{k, l-1}+\mu_{k+1, l-1})}
 {s(\lambda_{k, l-1}+\mu_{k+1, l})}.
 \end{align*}
Hence we obtain
\begin{align*}
-\sum_{m=j+1}^{n}B_{k, m}-E_k-F_k 
=e_{k, n}\tilde{e}_{k+1, n}\frac{s(\mu_j)}
{s(\lambda_{k,j-1}+\mu_{k+1,j-1})}\prod_{l=j}^{n}
\frac{ s(\lambda_{k, l-1}+\mu_{k+1, l-1})}
 {s(\lambda_{k, l-1}+\mu_{k+1, l})}.
\end{align*}
Therefore, we have 
\begin{align*}
A_k-\sum_{m=j+1}^{n}B_{k, m}-E_k-F_k 
=&\frac{(-1)^{j-1}e_{k,n}\tilde{e}_{k+1,n}s(\mu_j)}
{s(\lambda_{k,j-1}+\mu_{k+1,j-1})}
 \frac{\prod_{l=1, l \neq k}^{n}s(\lambda_{l, k-1}+\mu_{l, k})}
 {\prod_{l=1, l \neq k}^{n+1}s(\lambda_{l, k-1}+\mu_{l+1, k})}
 \\
 &\times 
  \frac{-\tilde{e}_{n+1}(e_{k,n}\tilde{e}_{k+1,n}-e_{k,n}^{-1}\tilde{e}_{k+1,n}^{-1})
 +e_{k,n}\tilde{e}_{k+1,n+1}-e_{k,n}^{-1}\tilde{e}_{k+1,n+1}^{-1}}{2\sqrt{-1}}
 \\
 =&\frac{(-1)^{j-1}s(\mu_j)s(\mu_{n+1})}
{s(\lambda_{k,j-1}+\mu_{k+1,j-1})}
 \frac{\prod_{l=1, l \neq k}^{n}s(\lambda_{l, k-1}+\mu_{l, k})}
 {\prod_{l=1, l \neq k}^{n+1}s(\lambda_{l, k-1}+\mu_{l+1, k})}.
\end{align*}
Similarly, we can compute the second term in the right hand side of 
\eqref{eq_connection_formula_proof_j=1-n}. 
This completes the proof for the 
connection formula \eqref{eq_connection_formula_10}. 

The connection formula for $j=n+1$, namely, 
the connection formula \eqref{eq_connection_formula_10_n+1} 
can be proved in the same way above. By analytic continuation of the connection formula in Proposition \ref{[M2], prop 2.5} and Lemma \ref{lem_analytic_continuation_of_domain}, 
we get 
\begin{align*}
\tilde{D}_{n+1}^{(1)}=&
\sum_{ k=1}^{ n+1} (-1)^{n} \frac{s(\mu_{n+1})}{s(\lambda_{n+1, k-1}+\mu_{n+1, k})}
 \prod_{1 \leq l \leq n+1,\atop l \neq k} 
\frac{s(\lambda_{l, k-1}+\mu_{l, k})}{s(\lambda_{l, k-1}+\mu_{l+1, k})} e_{k, n}\tilde{e}_{k+1, n+1}\tilde{D}_k^{(0)}\nn
\\
&-\tilde{e}_{n+1}\{ t_n<z,\ 0<t_n<\cdots <t_1<1\}. 
\end{align*}
Then, in the integral associated with $\{ t_n<z,\ 0<t_n<\cdots <t_1<1\}$, 
the integral with respect to the variables $t_1,\ldots, t_{n-1}$ corresponds to 
the integral defined for $\tilde{D}_n^{(1)}$ of ${}_nF_{n-1}(t_n)$ because $0<t_n<1$. 
Hence we can apply the induction hypothesis to $\{ t_n<z,\ 0<t_n<\cdots <t_1<1\}$ 
and obtain a sum of integrals associated with $\{0<t_n<z,\ 0<t_k<\cdots<t_n,\ 1<t_1<\cdots<t_{k-1} \}$ 
($k=1,\ldots,n$), which is nothing but $\tilde{D}_k^{(0)}$ of ${}_{n+1}F_n(z)$. 

Therefore, we obtain 
\begin{align*}
\tilde{D}_{n+1}^{(1)}
=&\sum_{k=1}^{n} \left\{ (-1)^ne_{k, n}\tilde{e}_{k+1, n+1} \frac{s(\mu_{n+1})}{s(\lambda_{n+1, k-1}+\mu_{n+1, k})} \prod_{l=1, l \neq k}^{n+1} \frac{s(\lambda_{l, k-1}+\mu_{l, k})}
{s(\lambda_{l,k-1}+\mu_{l+1, k})} \right.
\\
&\left.+(-1)^n\tilde{e}_{n+1} \prod_{l=1, l \neq k}^{n} \frac{s(\lambda_{l, k-1}+\mu_{l, k})}
{s(\lambda_{l, k-1}+\mu_{l+1, k})} \right\}\tilde{D}_k^{(0)} \\
&+(-1)^{n-1}\frac{s(\mu_{n+1})}{s(\lambda_{n+1, n}+\mu_{n+1, n+1})} \prod_{l=1}^{n} \frac{s(\lambda_{l, n}+\mu_{l, n+1})}{s(\lambda_{l, n}+\mu_{l+1, n+1})}\tilde{D}_{n+1}^{(0)}. 
\end{align*}
Straightforward computation yields the connection formula \eqref{eq_connection_formula_10_n+1}. 
\end{proof}

\begin{rem}
The connection formulas \eqref{eq_connection_formula_10} 
and \eqref{eq_connection_formula_10_n+1} for $n=2$ 
was computed in \cite{Mimachi 3F2}. The 
connection formula \eqref{eq_connection_formula_10_n+1} 
was obtained in 
\cite{Kawabata 1},
\cite{Mimachi intersection numbers for n+1Fn}, 
\cite{Okubo Takano Yoshida}. 
\end{rem}

From Theorem \ref{thm_connection_formula_10}, we have the connection matrix 
$C^{(10)}$ given by 
\begin{equation*}
\left(\tilde{D}_1^{(1)},\ldots, \tilde{D}_{n+1}^{(1)}\right)=
\left(\tilde{D}_1^{(0)},\ldots,\tilde{D}_{n+1}^{(0)}\right)C^{(10)}.
\end{equation*}
The elements of $C^{(10)}$ are products of the sine and the cosecant:
\begin{align*}
\left(C^{(10)}\right)_{i,j}=&\frac{s(\beta_j-\alpha_j)s(\alpha_{n+1})}{s(\beta_i-\alpha_j)s(\beta_i-\alpha_{n+1})} \prod_{k=1, k \neq i}^{n+1}\frac{s(\alpha_k-\beta_i)}{s(\beta_k-\beta_i)}\quad 
(1\leq i\leq n+1,\ 1\leq j\leq n),
\\
 \left(C^{(10)}\right)_{i,n+1}=&\prod_{k=1, k \neq i}^{n+1}
 \frac{s(\alpha_k-\beta_i)}{s(\beta_k-\beta_i)}\quad 
(1\leq i\leq n+1). 
\end{align*}
We expect that there exists the inverse matrix $C^{(01)}$ of $C^{(10)}$, which implies
the set 
$\left\{\tilde{D}_1^{(1)},\ldots, \tilde{D}_{n+1}^{(1)}\right\}$ 
of the domains of integration 
gives a fundamental system of solutions at $z=1$, 
 and 
the elements of $C^{(01)}$ are products of the sine and the cosecant. 
Recall that $C^{(10)}$ are derived from analytic continuation of 
the connection formulas \eqref{eq_connection_formula_i0} in Proposition \ref{[M2], prop 2.5}, 
which relates the fundamental system of solutions at $z=\infty$ 
to that at $z=0$. 
However, we have not found a way to use the induction hypothesis 
to the analytic continuation along the path $\gamma$ of the 
 connection formula relating the fundamental system of solutions at $z=0$  
to that at $z=\infty$. Instead, we present the inverse matrix $C^{(01)}$ below 
and  give a proof by direct calculation. 

Let a square matrix $C^{(01)}$ of order $n+1$ be defined by 
\begin{align*}
\left(C^{(01)}\right)_{i,j}=&-
\frac{s(\alpha_i)s(\beta_{1, n} -\alpha_{1, n+1}-\beta_j +\alpha_i)s(\beta_j-\alpha_j)}{s(\alpha_{n+1})s(\beta_{1, n}-\alpha_{1, n+1})s(\beta_j-\alpha_i)}
\prod_{k=1, k \neq i}^{n} \frac{s(\beta_k - \alpha_i)}{s(\alpha_k-\alpha_i)}
\end{align*}
for $1\le i\le n$, $1\le j\le n+1$,  
and  
\begin{equation*}
\left(C^{(01)}\right)_{n+1,i}=-\frac{s(\beta_i-\alpha_i)}{s(\beta_{1, n}-\alpha_{1, n+1})}
\end{equation*}
for $1\le i\le n+1$. 
\begin{thm}\label{thm_connection_matrix_01}
Assume for $1\leq i,j \leq n+1$, 
$\alpha_i-\beta_j,\alpha_i-\alpha_j,\beta_i-\beta_j,\beta_{1, n} -\alpha_{1, n+1}\not\in\Z$.  
Then the matrix $C^{(01)}$ is the inverse of the connection matrix $C^{(10)}$. \qed 
\end{thm}
\begin{proof}
We compute the elements $\left(C^{(01)} C^{(10)}\right)_{i, j}$ of the product directly. 
We consider separately the following cases: (i) $1\leq i,j\leq n$,  
(ii) $1\leq i\leq n$, $j=n+1$, 
(iii) $i=n+1$, $1\leq j\leq n$, (iv) $i=j=n+1$.

In the case of (i), we have  
\begin{align*}
\left(C^{(01)} C^{(10)}\right)_{i, j}=&-\sum_{k=1}^{n+1}\frac{s(\alpha_i)s(\beta_{1, n}-\alpha_{1, n+1}-\beta_k+\alpha_i)s(\beta_k-\alpha_k)}{s(\alpha_{n+1})s(\beta_{1, n}-\alpha_{1, n+1})s(\beta_k-\alpha_i)}\\
& \times \frac{s(\beta_j-\alpha_j)s(\alpha_{n+1})}{s(\beta_k-\alpha_j)s(\beta_k-\alpha_{n+1})}\prod_{l=1, l \neq i}^{n} \frac{s(\beta_l-\alpha_i)}{s(\alpha_l-\alpha_i)} \prod_{l=1, l \neq k}^{n+1}\frac{s(\alpha_l-\beta_k)}{s(\beta_l-\beta_k)} \\
=&-\frac{s(\alpha_i)s(\beta_j-\alpha_j)}{s(\beta_{1, n}-\alpha_{1, n+1})}\prod_{l=1, l \neq i}^{n} \frac{s(\beta_l-\alpha_i)}{s(\alpha_l-\alpha_i)} \\
& \times \left( \sum_{k=1}^{n+1} \frac{s(\beta_{1, n}-\alpha_{1, n+1}-\beta_k+\alpha_i)s(\beta_k-\alpha_k)}{s(\beta_k-\alpha_i)s(\beta_k-\alpha_j)s(\beta_k-\alpha_{n+1})}\prod_{l=1, l \neq k}^{n+1}\frac{s(\beta_k-\alpha_l)}{s(\beta_k-\beta_l)} \right). 
\end{align*}
 Put $a_i=e^{\pi \sqrt{-1} \alpha_i}, b_i=e^{\pi \sqrt{-1} \beta_i}$. Then by definition of 
 the sine, we get 
\begin{align}
\label{eq_1_ab}
&\sum_{k=1}^{n+1}\frac{s(\beta_{1, n}-\alpha_{1, n+1}-\beta_k+\alpha_i)s(\beta_k-\alpha_k)}{s(\beta_k-\alpha_i)s(\beta_k-\alpha_j)s(\beta_k-\alpha_{n+1})}\prod_{l=1, l \neq k}^{n+1}\frac{s(\beta_k-\alpha_l)}{s(\beta_k-\beta_l)} 
\\
&=- 2\sqrt{-1}\frac{ a_j a_{n+1} }{a_1^2 \cdots a_{n+1}^2}\sum_{k=1}^{n+1}  \frac{(b_1^2 \cdots b_n^2 a_i^2 - a_1^2 \cdots a_{n+1}^2 b_k^2)}{(b_k^2-a_i^2)(b_k^2-a_j^2)(b_k^2-a_{n+1}^2)} \frac{\prod_{l=1}^{n+1} (b_k^2-a_l^2)}{\prod_{l=1, l \neq k}^{n+1} (b_k^2-b_l^2)}. \nn
\end{align}
It is sufficient to prove  that 
the right hand side above is equal to zero as a rational function of 
$a_i$ and $b_i$. 

Let the rational function $f_1(x)$ be defined by
\begin{align*}
f_1(x)=& -2\sqrt{-1} \frac{ a_i a_{n+1} }{a_1^2 \cdots a_{n+1}^2} \frac{(b_1^2 \cdots b_n^2 a_i^2 - a_1^2 \cdots a_{n+1}^2 x)}{(x-a_i^2)(x-a_j^2)(x-a_{n+1}^2)} \prod_{l=1}^{n+1} \frac{(x-a_l^2)}{ (x-b_l^2)} . 
\end{align*}
In the case of $j\neq i$,
 $f_1(x)$ has only simple poles $x=b_k^2$ ($k=1,\ldots, n+1$).
 The residues are given as  
\begin{align*}
\underset{x=b_k^2}{\rm Res}f_1(x)dx&=-2\sqrt{-1} \frac{ a_i a_{n+1} }{a_1^2 \cdots a_{n+1}^2} \frac{(b_1^2 \cdots b_n^2 a_i^2 - a_1^2 \cdots a_{n+1}^2 b_k^2)}{(b_k^2-a_i^2)(b_k^2-a_j^2)(b_k^2-a_{n+1}^2)} \frac{\prod_{l=1}^{n+1} (b_k^2-a_l^2)}{\prod_{l=1, l \neq k}^{n+1} (b_k^2-b_l^2)}.
\end{align*}
Since the sum of all residues of a rational function on $\mathbb{P}$ is equal to zero, 
the right hand side of \eqref{eq_1_ab} is zero. 

In the case of $i=j$, 
$f_1(x)$ has only simple poles $x=a_i^2$,  
 and $x=b_k^2$ ($k=1,\ldots, n+1$).
The sum of the residues at $x=b_k^2$ ($k=1,\ldots, n+1$) is 
equal to the right hand side of \eqref{eq_1_ab}.  
The residues at $x=a_i^2$ is computed as 
\begin{align*}
\underset{x=a_i^2}{\rm Res}f_1(x)dx=& -2\sqrt{-1} \frac{ a_i a_{n+1} }{a_1^2 \cdots a_{n+1}^2} \frac{(b_1^2 \cdots b_n^2 a_i^2 - a_1^2 \cdots a_{n+1}^2 a_i^2)}{(a_i^2-a_{n+1}^2)} \frac{\prod_{l=1, l \neq i}^{n+1} (a_i^2-a_l^2)}{\prod_{l=1}^{n+1} (a_i^2-b_l^2)}  \\
=& 2\sqrt{-1} \frac{ a_i^2 b_i }{a_1 \cdots a_{n+1} b_1 \cdots b_n} \frac{(b_1^2 \cdots b_n^2 - a_1^2 \cdots a_{n+1}^2)}{(b_i^2-a_i^2)(a_i^2-1)} \prod_{l=1, l \neq i}^{n} \frac{b_l (a_i^2-a_l^2)}{a_l (a_i^2-b_l^2)}  \\
=& \frac{s(\beta_{1, n}-\alpha_{1, n+1})}{s(\beta_i -\alpha_i)s(\alpha_i)}\prod_{l=1, l \neq i}^{n} \frac{s(\alpha_i -\alpha_l)}{s(\alpha_i - \beta_l)}.
\end{align*}
Hence, we have 
\begin{align*}
\text{R.H.S. of \eqref{eq_1_ab}}
=&\sum_{k=1}^{n+1} \underset{x=b_k^2}{\rm Res} f_1(x)dx \\
=& -\underset{x=a_i^2}{\rm Res}f_1(x)dx \\
=& -\frac{s(\beta_{1, n}-\alpha_{1, n+1})}{s(\beta_i -\alpha_i)s(\alpha_i)}\prod_{l=1, l \neq i}^{n} \frac{s(\alpha_i -\alpha_l)}{s(\alpha_i - \beta_l)}. 
\end{align*}
Therefore, we obtain $\left(C^{(01)} C^{(10)}\right)_{i, i}=1$  
for $1\le i, j\le n$. 

In the case of (ii), we have 
\begin{align*}
\left(C^{(01)} C^{(10)}\right)_{i, n+1}=&-\sum_{k=1}^{n+1} \frac{s(\alpha_i) s(\beta_{1, n}-\alpha_{1, n+1}-\beta_{k}+\alpha_{i})s(\beta_k-\alpha_k)}{s(\alpha_{n+1})s(\beta_{1, n}-\alpha_{1, n+1})s(\beta_k -\alpha_i)}
 \prod_{l=1, l \neq i}^{n} \frac{s(\beta_l-\alpha_i)}{s(\alpha_l-\alpha_i)}\prod_{l=1, l \neq k}^{n+1}\frac{s(\alpha_l-\beta_k)}{s(\beta_l-\beta_k)} \\
=&-\frac{s(\alpha_i)}{s(\alpha_{n+1})s(\beta_{1, n}-\alpha_{1, n+1})}\prod_{l=1, l \neq i}^{n} \frac{s(\beta_l-\alpha_i)}{s(\alpha_l-\alpha_i)} \\
&\times \sum_{k=1}^{n+1} \frac{s(\beta_{1, n}-\alpha_{1, n+1}-\beta_{k}+\alpha_{i})s(\beta_k-\alpha_k)}{s(\beta_k -\alpha_i)} \prod_{l=1, l \neq k}^{n+1}\frac{s(\beta_k-\alpha_l)}{s(\beta_k-\beta_l)}.  
\end{align*}
Rewriting the parameters $\alpha_i$, $\beta_i$ to $a_i$, $b_i$, we get 
\begin{align*}
&\sum_{k=1}^{n+1} \frac{s(\beta_{1, n}-\alpha_{1, n+1}-\beta_k +\alpha_i)s(\beta_k -\alpha_k)}{s(\beta_k -\alpha_i)} \prod_{l=1, l \neq k}^{n+1} \frac{s(\beta_k -\alpha_l)}{s(\beta_k -\beta_l)} \\
&=\sum_{k=1}^{n+1} \frac{1}{2 \sqrt{-1}} \frac{1}{a_1 \cdots a_{n+1} b_k^2} \frac{(b_1^2 \cdots b_n^2 a_i^2 - a_1^2 \cdots a_{n+1}^2 b_k^2)}{(b_k^2-a_i^2)} \frac{\prod_{l=1}^{n+1} (b_k^2-a_l^2)}{\prod_{l=1, l \neq k}^{n+1} (b_k^2 -b_l^2)}. 
\end{align*}
Let the rational function $f_2(x)$ be defined by 
\begin{align*}
f_2(x)=\frac{1}{2 \sqrt{-1}} \frac{1}{a_1 \cdots a_{n+1} x} \frac{(b_1^2 \cdots b_n^2 a_i^2 - a_1^2 \cdots a_{n+1}^2 x)}{(x-a_i^2)} \prod_{l=1}^{n+1} \frac{(x-a_l^2)}{(x -b_l^2)}.
\end{align*}
Then $f_2(x)$ has only simple poles $x=0$, $x=\infty$,   
 and $x=b_k^2$ ($k=1,\ldots, n+1$).
 The residues are given as 
 \begin{align*}
 \underset{x=0}{\rm Res}f_2(x)dx 
=& - \frac{1}{2 \sqrt{-1}},
\\
\underset{x=\infty}{\rm Res}f_2(x)dx=&\frac{1}{2 \sqrt{-1}}, 
\\
\underset{x=b_k^2}{\rm Res}f_2(x)dx=&\frac{1}{2 \sqrt{-1}} \frac{1}{a_1 \cdots a_{n+1} b_k^2} \frac{(b_1^2 \cdots b_n^2 a_i^2 - a_1^2 \cdots a_{n+1}^2 b_k^2)}{(b_k^2-a_i^2)} \frac{\prod_{l=1}^{n+1} (b_k^2-a_l^2)}{\prod_{l=1, l \neq k}^{n+1} (b_k^2 -b_l^2)}. 
\end{align*}
 Hence, we have 
 \begin{align*}
\sum_{k=1}^{n+1}\underset{x=b_k^2}{\rm Res}f_2(x)dx 
=&-\underset{x=0}{\rm Res}f_2(x)dx-\underset{x=\infty}{\rm Res}f_2(x)dx=0.
\end{align*}
 Therefore, 
 we obtain 
 $\left(C^{(01)} C^{(10)}\right)_{i, n+1}=0$ 
for $1\leq i\leq n$. 

In the case of (iii), we have 
\begin{align*}
\left(C^{(01)} C^{(10)}\right)_{n+1, j}=& -\frac{s(\beta_j -\alpha_j)s(\alpha_{n+1})}{s(\beta_{1, n}-\alpha_{1, n+1})}\left( \sum_{k=1}^{n+1} \frac{s(\beta_k- \alpha_k)}{s(\beta_k -\alpha_j)s(\beta_k -\alpha_{n+1})} \prod_{l=1, l \neq k}^{n+1} \frac{s(\beta_k-\alpha_l)}{s(\beta_k- \beta_l)} \right). 
\end{align*}
Rewriting the parameters $\alpha_i$, $\beta_i$ to $a_i$, $b_i$, we get
\begin{align*}
&\sum_{k=1}^{n+1} \frac{s(\beta_k- \alpha_k)}{s(\beta_k -\alpha_j)s(\beta_k -\alpha_{n+1})} \prod_{l=1, l \neq k}^{n+1} \frac{s(\beta_k-\alpha_l)}{s(\beta_k- \beta_l)}  \\
&= \sum_{k=1}^{n+1} 2 \sqrt{-1} \frac{b_1 \cdots b_n a_j a_{n+1}}{a_1 \cdots a_{n+1}} \frac{1}{(b_k^2-a_j^2)(b_k^2-a_{n+1}^2)} \frac{\prod_{l=1}^{n+1} (b_k^2-a_l^2)}{\prod_{l=1, l \neq k}^{n+1} (b_k^2-b_l^2)}. 
\end{align*}
Let the rational function $f_3(x)$ be defined by 
\begin{align*}
f_3(x)= 2 \sqrt{-1} \frac{b_1 \cdots b_n a_j a_{n+1}}{a_1 \cdots a_{n+1}} \frac{1}{(x-a_j^2)(x-a_{n+1}^2)} \prod_{l=1}^{n+1} \frac{ (x-a_l^2)}{ (x-b_l^2)}. 
\end{align*}
Then $f_3(x)$ has only simple poles $x=b_k^2$ ($k=1,\ldots, n+1$).
 The residues are given as 
\begin{align*}
\underset{x=b_k^2}{\rm Res}f_3(x)dx
=2 \sqrt{-1} \frac{b_1 \cdots b_n a_j a_{n+1}}{a_1 \cdots a_{n+1}} \frac{1}{(b_k^2-a_j^2)(b_k^2-a_{n+1}^2)} \frac{\prod_{l=1}^{n+1} (b_k^2-a_l^2)}{\prod_{l=1, l \neq k}^{n+1} (b_k^2-b_l^2)}. 
\end{align*}
Hence we obtain $\left(C^{(01)} C^{(10)}\right)_{n+1, j}=0$ 
for $1\leq j\leq n$.

In the case of (iv), we have 
\begin{align*}
\left(C^{(01)} C^{(10)}\right)_{n+1, n+1}=-\frac{1}{s(\beta_{1, n}-\alpha_{1, n+1})} \sum_{k=1}^{n+1} s(\beta_k -\alpha_k) \prod_{l=1, l \neq k}^{n+1} \frac{s(\beta_k -\alpha_l)}{s(\beta_k -\beta_l)}. 
\end{align*}
Rewriting the parameters $\alpha_i$, $\beta_i$ to $a_i$, $b_i$, we get 
\begin{align*}
&\sum_{k=1}^{n+1} s(\beta_k -\alpha_k) \prod_{l=1, l \neq k}^{n+1} \frac{s(\beta_k -\alpha_l)}{s(\beta_k -\beta_l)} 
= -\sum_{k=1}^{n+1} \frac{b_1 \cdots b_n}{2 \sqrt{-1} a_1 \cdots a_{n+1} b_k^2} \frac{\prod_{l=1}^{n+1} (b_k^2 -a_l^2)}{\prod_{l=1, l \neq k}^{n+1} (b_k^2 -b_l^2)}. 
\end{align*}
Let the rational function $f_4(x)$ be defined by 
\begin{align*}
f_4(x)=-\frac{b_1 \cdots b_n}{2 \sqrt{-1} a_1 \cdots a_{n+1}x} \prod_{l=1}^{n+1} \frac{x-a_l^2}{x-b_l^2}. 
\end{align*}
Then $f_4(x)$ has only simple poles $x=0$, $x=\infty$, and 
$x=b_k^2$ ($k=1,\ldots, n+1$).
 The residues are given as 
\begin{align*}
\underset{x=0}{\rm Res}f_4(x)dx 
=& - \frac{a_1 \cdots a_{n+1}}{2 \sqrt{-1} b_1 \cdots b_n},
\\
\underset{x=\infty}{\rm Res}f_4(x)dx
=& \frac{b_1 \cdots  b_n}{2 \sqrt{-1} a_1 \cdots a_{n+1}},
\\
\underset{x=b_k^2}{\rm Res}f_4(x)dx =& -\sum_{k=1}^{n+1} \frac{b_1 \cdots b_n}{2 \sqrt{-1} a_1 \cdots a_{n+1} b_k^2} \frac{\prod_{l=1}^{n+1} (b_k^2 -a_l^2)}{\prod_{l=1, l \neq k}^{n+1} (b_k^2 -b_l^2)}. 
\end{align*}
Hence, we have 
\begin{align*}
\sum_{k=1}^{n+1}\underset{x=b_k^2}{\rm Res}f_4(x)dx 
=& -\underset{x=0}{\rm Res}f_4(x)dx -\underset{x=\infty}{\rm Res}f_4(x)dx
 \\
=& -s(\beta_{1, n}-\alpha_{1, n+1}). 
\end{align*}
Therefore, we obtain $\left(C^{(01)} C^{(10)}\right)_{n+1,n+1}=1$.

\end{proof}

\begin{cor}
Assume for $1 \leq i,  j \leq n+1$
\begin{align*}
&\mathrm{Re}(\alpha_i-\beta_j+1)>0\quad (j\neq i),
\\
&\mathrm{Re}(\beta_j-\alpha_j)>0, 
\\
&\alpha_i-\alpha_j\notin \Z \quad (j\neq i),\quad 
\alpha_i-\beta_j \notin \Z, \quad \beta_i-\beta_j \notin \Z \quad  (j\neq i). 
\end{align*}
Then we have 
\begin{align*}
F_{\tilde{D}_{i}^{(0)}}(z)=&-\sum_{j=1}^{n}\frac{s(\alpha_j)s(\beta_{1, n} -\alpha_{1, n+1}-\beta_i +\alpha_j)s(\beta_i-\alpha_i)}{s(\alpha_{n+1})s(\beta_{1, n}-\alpha_{1, n+1})s(\beta_i-\alpha_j)}
\prod_{l=1, l \neq j}^{n} \frac{s(\beta_l - \alpha_j)}{s(\alpha_l-\alpha_j)}F_{\tilde{D}_{j}^{(1)}}(z)
 \\
&-\frac{s(\beta_i-\alpha_i)}{s(\beta_{1, n}-\alpha_{1, n+1})}F_{\tilde{D}_{n+1}^{(1)} }(z).
\end{align*}
Moreover, $\left\{ F_{\tilde{D}_1^{(1)}}(z),\ldots,F_{\tilde{D}_{n+1}^{(1)}}\right\}$ forms a fundamental system 
of solutions  to the generalized hypergeometric equation. \qed 
\end{cor}
 
We note that together with Proposition \ref{[M2], prop 2.1} and Theorem  
 \ref{thm_connection_formula_10}, we have the connection matrix $C^{(1\infty)}$ 
 expressing the fundamental system of solutions 
 $\left\{ F_{\tilde {D}^{(1)}_1}(z),\ldots, F_{\tilde{D}^{(1)}_1}(z)\right\}$ at $z=1$ 
 in terms of the fundamental system of solutions 
 $\left\{ F_{D^{(\infty)}_1}(z),\ldots, F_{D^{(\infty)}_1}(z)\right\}$ 
at $\infty$ as a product of the connection matrices. It is observed that the elements of 
the connection matrix $C^{(1\infty)}$ for $n\le 3$ are products of the sine and cosecant. 
However, there are elements of the inverse of $C^{(1\infty)}$ for $n\le 3$ which are not 
products of the sine and cosecant. If we want to have  
connection matrices between the fundamental system of solutions of 
singular points $z=1$ and $z=\infty$ whose elements are products of the sine and cosecant, 
then we should perform the change of variable $z=1/w$ to the connection formulas 
\eqref{eq_connection_formula_10} and \eqref{eq_connection_formula_10_n+1}. 
Due to Theorem \ref{thm_connection_matrix_01}, all elements of 
both the connection matrix and its inverse are products of the sine and cosecant. 
 
\subsection{Periodicity} 
 
In this subsection, we point out that the connection matrices 
  multiplied by diagonal matrices in the previous sections
are invariant under integer shifts of the
 characteristic exponents $\alpha_i$, $\beta_i$ of the generalized 
 hypergeometric equation. 
This is true for the connection matrix with the fundamental 
systems $X_0$ and $X_\infty$ of solutions at $z=0$ and $z=\infty$, 
and it was used for the 
construction of the monodromy-invariant general solutions of Fuchsian systems  of rank $2$ 
in \cite{Iorgov Lisovyy Teschner}  and rank $N\geq 2$  
in \cite{GIL2} with help of the $W_N$ conformal field theory with 
 central charge $c=N-1$. 
When $N=2$, $W_N$-algebra is the Virasoro algebra. 
As a result, for $N=2$ the tau function of the sixth Painlev\'e equation 
is obtained in terms of the Virasoro conformal blocks with central charge 
$c=1$ 
\cite{Iorgov Lisovyy Teschner}, and for $N\ge 2$ 
the isomonodromic 
tau function of the so-called Fuji-Suzuki-Tsuda system \cite{FS}, \cite{Suzuki}, \cite{Tsuda1} is obtained 
in terms of 
semi-degenerate conformal blocks of $W_N$-algebra
 with central charge $c=N-1$, which are equivalent to the Nekrasov partition functions \cite{Nekrasov} 
by AGT correspondence \cite{AGT}, \cite{Wyllard}. The series representation of the tau function of 
the Fuji-Suzuki-Tsuda system is also derived by expanding the Fredholm determinant \cite{GIL3}. 
 The periodicity of the connection matrix 
 of $q$-hypergeometric series was used in \cite{Jimbo Nagoya Sakai} in order to obtain 
 a fundamental solution to 
 the connection-preserving deformation of the $2$ by $2$ $q$-difference linear system 
 associated with the $q$-difference Painlev\'e VI equation.

 We take 
 \begin{align*}
D_i^{(0)}=&\{t\in T_\R \mid z<t_n< \cdots <t_i <0,\ 1<t_1< \cdots < t_{i-1} \} \quad (1 \leq i \leq  n+1),
\\
D_i^{(\infty)}=&\{t\in T_\R \mid t_i< \cdots <t_n <z,\ 0<t_{i-1}< \cdots < t_1< 1 \} \quad (1 \leq i \leq  n+1).
\end{align*}
 as domains of integration giving the fundamental systems of solutions at $z=0$ and $z=\infty$, and 
  \begin{align*}
\tilde{D}_i^{(1)}&=\{t\in T_\R | t_i< \cdots <t_n <0,\ 0<t_{i-1}< \cdots < t_1< 1 \} \quad (1 \leq i \leq  n), 
 \\
\tilde{D}_{n+1}^{(1)}&=\{t\in T_\R | z<t_n< \cdots < t_1< 1 \}
\end{align*}
 as domains of integration giving 
 the fundamental system of solutions at $z=1$.

From Proposition \ref{[M2], prop 2.5}, we have the connection 
matrix $C^{(\infty 0)}$ given by
\begin{equation*}
\left(D_1^{(\infty)},\ldots, D_{n+1}^{(\infty)}\right)=
\left(D_1^{(0)},\ldots,D_{n+1}^{(0)}\right)C^{(\infty 0)}.
\end{equation*} 
The elements of $C^{(\infty 0)}$ read as
\begin{align*}
\left( C^{(\infty 0)}\right)_{i,j}=
\frac{s(\beta_j-\alpha_j)}{s(\beta_i-\alpha_j)}
\prod_{k=1, k \neq i}^{n+1}
 \frac{s(\alpha_k-\beta_i)}{s(\beta_k-\beta_i)} 
\end{align*} 
 for $1\leq i,j\leq n+1$. 
 
Let the diagonal matrices $N_0$, $N_1$, $N_\infty$ 
of order $n+1$ be defined by  
 \begin{align*}
 N_0=&{\rm diag} (e(\alpha_1),\ldots,e(\alpha_{n+1})),
 \\
 N_1=&e(\beta_{1,n}-\alpha_{1,n+1})
 {\rm diag}(e(\beta_1),\ldots,e(\beta_{n+1})),
 \\
 N_\infty=&N_1. 
 \end{align*}
\begin{prop} 
The connection matrices associated with 
 the fundamental systems of solutions
 \begin{align*}
 \left(D_1^{(0)},\ldots,D_{n+1}^{(0)}\right)N_0,
 \quad
 \left(\tilde{D}_1^{(1)},\ldots, \tilde{D}_{n+1}^{(1)}\right)N_1,
 \quad 
 \left(D_1^{(\infty)},\ldots, D_{n+1}^{(\infty)}\right)N_\infty
 \end{align*}
 are invariant under integer shifts of the
 characteristic exponents of the generalized hypergeometric equation. \qed  
\end{prop}
 \begin{proof}
 It is sufficient to consider the shifts $\alpha_i\mapsto \alpha_i\pm 1$ and   
 $\beta_i\mapsto \beta_i\pm 1$ for $1\leq i\leq n+1$. 
 The connection matrices which express the fundamental system of 
 solutions at  $z=1$, $z=\infty$ in terms of the fundamental system of 
 solutions at  $z=\infty$ is 
\begin{align*}
\widehat{C}^{(10)}=N_0^{-1} {\rm diag}(-e(\beta_1),\ldots,-e(\beta_{n+1}))C^{(10)}N_1, 
\quad 
\widehat{C}^{(\infty 0)}=N_0^{-1}C^{(\infty 0)}N_\infty, 
\end{align*} 
 respectively, 
 because of Lemma 
 \ref{lem_analytic_continuation_of_domain} and the relations of the 
 parameters \eqref{eq_transformation_lambda_alpha_beta}.  
 Since the matrices $C^{(10)}$, $C^{(\infty 0)}$ are 
 explicitly written, it is immediate to check that 
 the shift $\alpha_i\mapsto \alpha_i\pm 1$ or   
 $\beta_i\mapsto \beta_i\pm 1$ preserves  
the connection matrices $\widehat{C}^{(1 0)}$, 
$\widehat{C}^{(\infty 0)}$. 
 \end{proof}

 \textbf{Acknowledgments.}
This work is partially supported by JSPS KAKENHI Grant Number JP18K03326.

\end{document}